\documentclass[11pt]{article}

\usepackage{amssymb,amsmath,amsfonts,amsthm}
\usepackage{latexsym}
\usepackage{graphics}
\usepackage{indentfirst}
\newtheorem*{thm*}{Theorem}
\newtheorem{thm}{Theorem}
\newtheorem{lem}[thm]{Lemma}
\newtheorem{pro}[thm]{Proposition}

\newtheorem{cor}[thm]{Corollary}

\newtheorem{ques}[thm]{Question}

\newcommand{\N}{\mathbb{N}}

\begin{document}

\title{Proportional Choosability of Complete Bipartite Graphs}

\author{Jeffrey A. Mudrock\footnotemark[1], Jade Hewitt\footnotemark[1], Paul Shin\footnotemark[1], and Collin Smith\footnotemark[1]}

\footnotetext[1]{Department of Mathematics, College of Lake County, Grayslake, IL 60030.  E-mail:  {\tt {jmudrock@clcillinois.edu}}}

\maketitle

\begin{abstract}
Proportional choosability is a list analogue of equitable coloring that was introduced in 2019.  The smallest $k$ for which a graph $G$ is proportionally $k$-choosable is the \emph{proportional choice number} of $G$, and it is denoted $\chi_{pc}(G)$.  In the first ever paper on proportional choosability, it was shown that when $2 \leq n \leq m$, $ \max\{ n + 1, 1 + \lceil m / 2 \rceil\} \leq \chi_{pc}(K_{n,m}) \leq n + m - 1$.  In this note we improve on this result by showing that $ \max\{ n + 1, \lceil n / 2 \rceil + \lceil m / 2 \rceil\} \leq \chi_{pc}(K_{n,m}) \leq n + m -1- \lfloor m/3 \rfloor$.  In the process, we prove some new lower bounds on the proportional choice number of complete multipartite graphs.  We also present several interesting open questions.

\medskip



\end{abstract}

\section{Introduction}\label{intro}

In this note all graphs are nonempty, finite, simple graphs unless otherwise noted.  Generally speaking we follow West~\cite{W01} for terminology and notation.  The set of natural numbers is $\N = \{1,2,3, \ldots \}$.  For $m \in \N$, we write $[m]$ for the set $\{1, \ldots, m \}$.  If $G$ is a graph and $S \subseteq V(G)$, we use $G[S]$ for the subgraph of $G$ induced by $S$.  We also let $N_G(S)$ be the subset of $V(G)$ consisting of all vertices adjacent (in $G$) to at least one vertex in $S$.  We write  $\Delta(G)$ for the maximum degree of a vertex in $G$.  We write $K_{n,m}$ for the equivalence class consisting of complete bipartite graphs with partite sets of size $n$ and $m$, and we write $K_{n_1, \ldots, n_t}$ for the equivalence class consisting of complete $t$-partite graphs with partite sets of size $n_1, \ldots, n_t$.  

In 2019, a new notion combining the notions of equitable coloring and list coloring called proportional choosability was introduced~\cite{KM19}.  In this note, we study the proportional choosability of complete bipartite graphs.  Before introducing proportional choosability, we briefly review equitable coloring and list coloring.  

\subsection{Equitable Coloring and List Coloring} \label{basic}

In the classical vertex coloring problem, we seek to color the vertices of a graph with up to $k$ colors so that adjacent vertices receive different colors, a so-called \emph{proper $k$-coloring}.  The chromatic number of a graph $G$, denoted $\chi(G)$, is the smallest $k$ such that $G$ has a proper $k$-coloring.  Equitable coloring is a variation on the classical vertex coloring problem that began with a conjecture of Erd\H{o}s~\cite{E64} in~1964.  This conjecture of Erd\H{o}s was proved in 1970~\cite{HS70} (Theorem~\ref{thm: HS} below).  In 1973, Meyer~\cite{M73} formally introduced equitable coloring.  An \emph{equitable $k$-coloring} of a graph $G$ is a proper $k$-coloring $f$ of $G$ such that the sizes of the color classes associated with $f$ differ by at most one.~\footnote{We work under the assumption that a proper $k$-coloring has exactly $k$, possibly empty, color classes.}  It is easy to see that for an equitable $k$-coloring, the color classes associated with the coloring are each of size $\lceil |V(G)|/k \rceil$ or $\lfloor |V(G)|/k \rfloor$.  We say that a graph $G$ is \emph{equitably $k$-colorable} if there exists an equitable $k$-coloring of $G$. Equitable colorings are useful when it is preferable to form a proper coloring without under-using or over-using any color (see~\cite{JR02,KJ06,P01,T73} for applications).

Unlike ordinary graph coloring, increasing the number of colors can make equitable coloring more difficult.  For example, $K_{3,3}$ is equitably 2-colorable, but it is not equitably 3-colorable.  In 1970, Hajn\'{a}l and Szemer\'{e}di proved the following result.

\begin{thm}[\cite{HS70}] \label{thm: HS}
Every graph $G$ has an equitable $k$-coloring when $k \geq \Delta(G)+1$.
\end{thm}

List coloring is another well-known variation on the classical vertex coloring problem, and it was introduced independently by Vizing~\cite{V76} and Erd\H{o}s, Rubin, and Taylor~\cite{ET79} in the 1970s.  For list coloring, we associate a \emph{list assignment} $L$ with a graph $G$ such that each vertex $v \in V(G)$ is assigned a list of colors $L(v)$ (we say $L$ is a list assignment for $G$).  The graph $G$ is \emph{$L$-colorable} if there exists a proper coloring $f$ of $G$ such that $f(v) \in L(v)$ for each $v \in V(G)$ (we refer to $f$ as a \emph{proper $L$-coloring} of $G$).  A list assignment $L$ is called a \emph{$k$-assignment} for $G$ if $|L(v)|=k$ for each $v \in V(G)$.  The \emph{list chromatic number} of a graph $G$, denoted $\chi_\ell(G)$, is the smallest $k$ such that $G$ is $L$-colorable whenever $L$ is a $k$-assignment for $G$.  We say $G$ is \emph{$k$-choosable} if $k \geq \chi_\ell(G)$.

Suppose that $L$ is a list assignment for a graph $G$.  The \emph{palette of colors associated with $L$} is $\cup_{v \in V(G)} L(v)$.  From this point forward, we use $\mathcal{L}$ to denote the palette of colors associated with $L$ whenever $L$ is a list assignment.  We say that $L$ is a \emph{constant $k$-assignment} for $G$ when $L$ is a $k$-assignment for $G$ and $|\mathcal{L}|=k$ (i.e., $L$ assigns the same list of $k$ colors to every vertex in $V(G)$).  Since $G$ must be $L$-colorable whenever $L$ is a constant $\chi_\ell(G)$-assignment for $G$, it is clear that $\chi(G) \leq \chi_\ell(G)$.

Since the focus of this note is the proportional choosability of complete bipartite graphs, it is worth briefly addressing the choosability of complete bipartite graphs.  First, it is well-known that complete bipartite graphs can be used to demonstrate that the list chromatic number and chromatic number of a graph can be arbitrarily far apart.  Indeed, $\chi(K_{n,m})=2$ and $\chi_\ell(K_{n,m}) = n+1$ whenever $m \geq n^n$.  It is also known that $(1/2 - o(1))\log_{2}(n) \leq \chi_\ell(K_{n,n}) \leq 2 + \log_{2}(n)$ where the $o(1)$-term tends to zero as $n$ tends to infinity (see~\cite{A92} and~\cite{A00}).  On the other hand, finding the exact list chromatic number of complete bipartite graphs is notoriously difficult.  In fact, the 3-choosable complete bipartite graphs were not fully characterized until 1995~\cite{O95} about 20 years after the introduction of list coloring (see section 8.2 in~\cite{W20} for some discussion).  Finally, there are many examples in the literature that study variants of list coloring and equitable variants of list coloring on complete bipartite graphs (see e.g.,~\cite{FK16, KN06, M18, M182, MC19}). 
 
\subsection{Proportional Choosability}

Kostochka, Pelsmajer, and West~\cite{KP03} introduced a list analogue of equitable coloring called equitable choosability in 2003, and this notion has received attention in the literature.  If $L$ is a $k$-assignment for a graph $G$, a proper $L$-coloring of $G$ is an \emph{equitable $L$-coloring} of $G$ if each color in $\mathcal{L}$ appears on at most $\lceil |V(G)|/k \rceil$ vertices.  We say $G$ is \emph{equitably $k$-choosable} if an equitable $L$-coloring of $G$ exists whenever $L$ is a $k$-assignment for $G$.  Importantly, notice that the definition of equitable $L$-coloring does not place a lower bound on how many times a color must be used, whereas in an equitable $k$-coloring there is both an upper and lower bound (i.e., $\lceil |V(G)|/k \rceil$ and $\lfloor |V(G)|/k \rfloor$) on how many times each color must be used.

Kaul, Pelsmajer, Reiniger, and the first author~\cite{KM19} introduced a new list analogue of equitable coloring called proportional choosability which places both an upper and lower bound on how many times a color must be used in a list coloring.  The study of proportional choosability is in its infancy, but it has received some attention in the literature (see~\cite{KM19, KM20, MP20}).  Now, following~\cite{KM19} we introduce the specifics.  Suppose that $L$ is a $k$-assignment for a graph $G$.  For each color $c \in \mathcal{L}$, the \emph{multiplicity of $c$ in $L$} is the number of vertices $v$ whose list $L(v)$ contains $c$.  The multiplicity of $c$ in $L$ is denoted by $\eta_L(c)$ (or simply $\eta(c)$ when the list assignment is clear).  So, $\eta_L(c)=\left\lvert{\{v \in V(G) : c \in L(v) \}}\right\rvert$.  A proper $L$-coloring $f$ of $G$ is a \emph{proportional $L$-coloring} of $G$ if for each $c \in \mathcal{L}$, $f^{-1}(c)$, the color class of $c$, is of size
$$ \left \lfloor \frac{\eta(c)}{k} \right \rfloor \; \; \text{or} \; \; \left \lceil \frac{\eta(c)}{k} \right \rceil.$$
We say $G$ is \emph{proportionally $L$-colorable} if a proportional $L$-coloring of $G$ exists, and we say $G$ is \emph{proportionally $k$-choosable} if $G$ is proportionally $L$-colorable whenever $L$ is a $k$-assignment for $G$.  Proportional choosability has some notable (perhaps surprising) properties.

\begin{pro} [\cite{KM19}] \label{pro: motivation}
If $G$ is proportionally $k$-choosable, then $G$ is both equitably $k$-choosable and equitably $k$-colorable.
\end{pro}

\begin{pro} [\cite{KM19}] \label{pro: monoink}
If $G$ is proportionally $k$-choosable, then $G$ is proportionally $(k+1)$-choosable.
\end{pro}

The \emph{proportional choice number} of a graph $G$, denoted $\chi_{pc}(G)$, is the smallest integer $k$ such that $G$ is proportionally $k$-choosable.  In light of Proposition~\ref{pro: monoink}, we know that a graph is proportionally $k$-choosable if and only if $k \geq \chi_{pc}(G)$.

\begin{pro} [\cite{KM19}] \label{lem: monotone}
Suppose $H$ is a subgraph of $G$.  If $G$ is proportionally $k$-choosable, then $H$ is proportionally $k$-choosable.  Consequently, $\chi_{pc}(G) \geq \chi_{pc}(H)$.
\end{pro}

Notice that Propositions~\ref{pro: monoink} and~\ref{lem: monotone} are particularly interesting since they do not hold in the contexts of equitable coloring and equitable choosability.  The purpose of this note is to study the proportional choice number of complete bipartite graphs.  In~\cite{KM19} it is shown that for any $m \in \N$, $\chi_{pc}(K_{1,m}) = 1 + \lceil m/2 \rceil$ and $\chi_{pc}(K_{m,m}) > m$.  Furthermore, it is shown that for any non-complete graph $G$, $\chi_{pc}(G) \leq |V(G)|-1$.  Combining these results along with Proposition~\ref{lem: monotone}, we obtain the following result.

\begin{thm} [\cite{KM19}] \label{thm: old}
Suppose $n, m \in \N$ satisfy $2 \leq n \leq m$.  Then,
$$\max \left \{ n+1, 1 + \left \lceil \frac{m}{2} \right \rceil \right \} \leq \chi_{pc}(K_{n,m}) \leq n + m - 1.$$
\end{thm}

Note that we assume $n \geq 2$ in the statement of Theorem~\ref{thm: old} since the proportional choice number of stars is known.  In this note we will give an improvement on both the upper and lower bound in Theorem~\ref{thm: old}.

\subsection{Summary of Results and Open Questions}

In Section~\ref{main}, we prove the following.

\begin{thm} \label{thm: main}
Suppose $n, m \in \N$ satisfy $2 \leq n \leq m$.  Then,
$$\max \left \{ n+1, \left \lceil \frac{n}{2} \right \rceil + \left \lceil \frac{m}{2} \right \rceil \right \} \leq \chi_{pc}(K_{n,m}) \leq n+m-1 - \left \lfloor \frac{m}{3} \right \rfloor.$$
\end{thm}

With the properties of the list chromatic number of complete bipartite graphs mentioned at the end of Subsection~\ref{basic} in mind, Theorem~\ref{thm: main} demonstrates that the proportional choice number of complete bipartite graphs behaves quite differently than the list chromatic number.  We prove Theorem~\ref{thm: main} by first proving the lower bound.  To prove the lower bound, we prove the following lower bound on the proportional choice number of complete $t$-partite graphs.

\begin{thm} \label{pro: tpartitelower}
Suppose $G = K_{n_1, \ldots, n_t}$ with $t, n_1, \ldots, n_t \in \N$ and $t \geq 2$. Suppose $s = \sum_{i=1}^t \lceil n_i/2 \rceil$. Then, $G$ is not proportionally $(s-1)$-choosable. Consequently, $\chi_{pc}(K_{n_1,\ldots,n_t}) \geq \sum_{i=1}^t \lceil n_i/2 \rceil$.
\end{thm}

The lower bound in Theorem~\ref{thm: main} is then implied by Theorems~\ref{thm: old} and~\ref{pro: tpartitelower}.  We also demonstrate that the bound in Theorem~\ref{pro: tpartitelower} does not hold with equality for certain complete multipartite graphs by proving the following.

\begin{pro} \label{pro: even-complete-t-partite}
Suppose $t, n_1, \ldots, n_t \in \N$, $t \geq 2$, and $n_i$ is even for each $i \in [t]$. Let $s = \sum_{i=1}^t n_i/2$. If $\max_{i \in [t]} n_i \leq s$, then $K_{n_1, \ldots, n_t}$ is not proportionally $s$-choosable. Consequently, $\chi_{pc}(K_{n_1, \ldots, n_t}) \geq 1 + \sum_{i=1}^t n_i/2$.
\end{pro}

It is worth mentioning that Proposition~\ref{pro: even-complete-t-partite} generalizes a result in~\cite{KM19} which says that $\chi_{pc}(K_{2*m}) \geq m+1$ whenever $m \geq 2$, where $K_{2*m}$ denotes the equivalence class of complete $m$-partite graphs where each partite set is of size 2.   

We finish Section~\ref{main} by using matching theory and a result from~\cite{KM19} to prove the upper bound in Theorem~\ref{thm: main}.  Theorem~\ref{thm: main} says something about an open question related to Theorem~\ref{thm: HS}.

\begin{ques} [\cite{KM19}] \label{ques: HS}
For any graph $G$, is $G$ proportionally $k$-choosable whenever $k \geq \Delta(G) + 1$? 
\end{ques}

Theorem~\ref{thm: main} tells us that $K_{3,m}$ is proportionally $(m+1)$-choosable whenever $m \geq 3$ (notice that this is not implied by Theorem~\ref{thm: old}).  More generally, Theorem~\ref{thm: main} implies that the answer to Question~\ref{ques: HS} is yes when we restrict our attention to complete bipartite graphs $K_{n,m}$ satisfying $2 \leq n \leq m$ and $n \leq \lfloor m/3 \rfloor + 2$.

Interestingly, we have not found any examples for which our lower bound in Theorem~\ref{thm: main} does not hold with equality.  Consequently, the following question is open.

\begin{ques} \label{ques: lower}
Is $\chi_{pc}(K_{n,m}) = \max \left \{ n+1, \left \lceil n/2 \right \rceil + \left \lceil m/2 \right \rceil \right \}$ whenever $n,m \in \N$ and $m \geq n$?
\end{ques}  

The following related question about the proportional choice number of complete multipartite graphs is also open.

\begin{ques} \label{ques: multipartite}
Suppose $G = K_{n_1, \ldots, n_t}$ with $t, n_1, \ldots, n_t \in \N$ and $t \geq 2$.  If $G$ does not satisfy the hypotheses of Proposition~\ref{pro: even-complete-t-partite}, is it the case that $\chi_{pc}(G) = \sum_{i=1}^t \lceil n_i/2 \rceil$?  If $G$ satisfies the hypotheses of Proposition~\ref{pro: even-complete-t-partite}, is it the case that $\chi_{pc}(G) = 1 + \sum_{i=1}^t  n_i/2 $? 
\end{ques}    

The following question about the asymptotics of $\chi_{pc}(K_{n,m})$ is also open.

\begin{ques} \label{ques: limit}
Suppose $n$ is a fixed natural number.  Is it the case that 
$$ \lim_{m \rightarrow \infty} \frac{\chi_{pc}(K_{n,m})}{m} = \frac{1}{2} ?$$
\end{ques}

It should be noted that since $\chi_{pc}(K_{1,m}) = 1 + \lceil m/2 \rceil$, we know that the answer to both Question~\ref{ques: lower} and Question~\ref{ques: limit} is yes when $n=1$.  Also, Theorem~\ref{thm: main} implies that for fixed $n \in \N$, $1/2 \leq \liminf_{m \rightarrow \infty} \chi_{pc}(K_{n,m})/m \leq \limsup_{m \rightarrow \infty} \chi_{pc}(K_{n,m})/m \leq 2/3.$   

\section{Proof of Theorem~\ref{thm: main}} \label{main}

\subsection{Lower Bound}

We begin by proving Theorem~\ref{pro: tpartitelower}.  First, we need a result from 1994 on the equitable coloring of complete multipartite graphs.

\begin{thm} [\cite{W94}] \label{thm: wu}
For any $K_{n_1,\ldots,n_t}$ with $t, n_1, \ldots, n_t \in \N$, let $p = n_1 + n_2 + \ldots + n_t$. Then $K_{n_1,n_2,\ldots,n_t}$ is equitably $s$-colorable if and only if either $s > p$ or $n_i \geq \lceil n_i / \lceil p / s \rceil \rceil \lfloor p/s \rfloor$ for all $i \in [t]$ and $\sum_{i=1}^t \lfloor n_i / \lfloor p / s \rfloor \rfloor \geq s \geq \sum_{i=1}^t \lceil n_i / \lceil p / s \rceil \rceil$ when $s \leq p$.
\end{thm}

We are now ready to prove Theorem~\ref{pro: tpartitelower}. 

\begin{proof}
For all $i\in[t]$, we let
$$
n'_i = \left\{
        \begin{array}{ll}
            n_i & \quad \text{if } n_i \text{ is odd} \\
            n_i - 1 & \quad \text{if } n_i \text{ is even.}
        \end{array}
    \right.
$$
Note that $s = \sum_{i=1}^t \lceil n_i / 2 \rceil = \sum_{i=1}^t \lceil n'_i / 2 \rceil$. Consider the graph $G' = K_{n'_1, \ldots, n'_t}$. Note that since $n'_i$ is odd for all $i \in [t]$, we have that $|V(G')| = 2s-t$.  Since $t \geq 2$, we know $|V(G')| \leq 2(s - 1)$ which implies that $\lceil |V(G')| / (s-1) \rceil \leq 2$. As such, $\sum_{i=1}^t \lceil n'_i / \lceil |V(G')| / (s-1) \rceil \rceil \geq \sum_{i=1}^t \lceil n'_i / 2 \rceil = s > s - 1$. Theorem~\ref{thm: wu} tells us $G'$ is not equitably $(s-1)$-colorable.  Proposition~\ref{pro: motivation} then implies that $G'$ is not proportionally $(s-1)$-choosable. As $G$ contains a copy of $K_{n'_1,\ldots,n'_t}$ as a subgraph, Proposition~\ref{lem: monotone} implies $G$ is not proportionally $(s-1)$-choosable.
\end{proof}

Notice that the lower bound in Theorem~\ref{thm: main} is implied by Theorems~\ref{thm: old} and~\ref{pro: tpartitelower}.  We end this Subsection by proving Proposition~\ref{pro: even-complete-t-partite} which demonstrates that the lower bound in Theorem~\ref{pro: tpartitelower} does not hold with equality for all complete multipartite graphs.

\begin{proof}
Let $G = K_{n_1, \ldots, n_t}$ with partite sets $A_i = \{u_{i,1},\ldots, u_{i,n_i}\}$ for each $i \in [t]$. We will now construct an $s$-assignment $L$ for $G$ for which there is no proportional $L$-coloring. Let $L$ be defined by $L(u_{i,j}) = [s-1] \cup \{s - 1 + i\}$ for each $i \in [t]$ and $j \in [n_i]$. Since $\max_{i \in [t]} n_i \leq s$, it follows that $\eta(c) \leq s$ for each $c \in \{s, \ldots, s - 1 + t\}$. Thus, a proportional $L$-coloring of $G$ may not use any color $c \in \{s, \ldots, s - 1 + t\}$ more than once. Moreover, we have that $\eta(c) = 2s$ for each $c \in [s-1]$, and so a proportional $L$-coloring of $G$ must use each color $c \in [s-1]$ exactly twice.

For the sake of contradiction, suppose that $f$ is a proportional $L$-coloring of $G$. Notice that $|V(G)| = 2s$ and $|f^{-1}([s-1])| = 2s - 2$, and so $|f^{-1}(\mathcal{L} - [s - 1])| = 2$. Also, for each color $c \in [s - 1]$, we have that $f^{-1}(c) \subseteq A_i$ for some $i \in [t]$; otherwise, we would have that vertices $u, v \in f^{-1}(c)$ are adjacent, which would imply that $f$ is not proper. In addition, we claim that $f^{-1}(\mathcal{L} - [s-1]) \subseteq A_j$ for some $j \in [t]$. To see this, notice that if $u, v \in f^{-1}(\mathcal{L} - [s-1])$ were in different partite sets $A_k, A_\ell$ respectively, then $A_{k} - \{u\}$, which is a subset of $f^{-1}([s-1])$, would be of odd size $n_k - 1$. However, this would imply that there exists a color $c \in f(A_k - \{u\})$ such that $|f^{-1}(c)| = 1$, which is a contradiction. Hence, if we let $u, v$ be the elements of $f^{-1}(\mathcal{L} - [s-1])$, then we have that $L(u) = L(v) = [s-1] \cup \{s - 1 + j\}$ for some $j \in [t]$. However, this implies that $f(u) = f(v) = s - 1 + j$. Thus, it follows that $|f^{-1}(s-1+j)| = 2$, which is a contradiction. Hence, there is no proportional $L$-coloring of $G$, and we have that $G$ is not proportionally $s$-colorable.
\end{proof}

\subsection{Upper Bound}

We begin with a definition.  Suppose $L$ is a $k$-assignment for a graph $G$. We say $c \in \mathcal{L}$ is a \emph{color of high multiplicity with respect to} $L$~\footnote{We sometimes omit ``with respect to $L$" when the list assignment is clear from context.} if $\eta(c) > k$.  Intuitively speaking, our strategy for proving the upper bound in Theorem~\ref{thm: main} is as follows: first take care of list assignments with no colors of high multiplicity, then take care of list assignments with many colors of high multiplicity, and finally take care of all remaining list assignments.

Dealing with list assignments with no colors of high multiplicty will require some matching theory.  So, we review some concepts.  A \emph{matching} in a graph $G$ is a set of edges with no shared endpoints.  If $M$ is a matching in $G$, and $X \subseteq V(G)$ such that each vertex in $X$ is an endpoint of an edge in $M$, we say that $X$ is \emph{saturated} by $M$.    The following classical result was proved by Hall in 1935.

\begin{thm}[\cite{H35}, {\bf Hall's Theorem}] \label{thm: Halls}
Suppose $B$ is a bipartite multigraph with bipartition $X, Y$.  Then, $B$ has a matching that saturates $X$ if and only if $|N_B(S)| \geq |S|$ for all $S \subseteq X$.
\end{thm}

In order to take care of list assignments with no colors of high multiplicity, we need two lemmas.  The first of these lemmas follows easily from Hall's Theorem.

\begin{lem} \label{lem: leq_k}
Suppose $L$ is a $k$-assignment for a graph $G$. Suppose that, for each color $c \in \mathcal{L}$, $\eta(c) \leq k$. Then, there is a proper $L$-coloring of $G$ with range of size $|V(G)|$.
\end{lem}

\begin{proof}
Suppose $G$ is a graph with vertices $v_1, \ldots, v_n$, and suppose $L$ is a $k$-assignment for $G$ such that, for each color $c \in \mathcal{L}$, $\eta(c) \leq k$. Now, consider an $X,Y$-bigraph $H$ defined as follows: $X = V(G)$, $Y =\mathcal{L}$, and for each $u \in X$ and $v \in Y$, $uv \in E(H)$ if and only if $v \in L(u)$.

Suppose $S$ is a nonempty subset of $X$. Let $E$ be the set of edges with one endpoint in $S$ and the other endpoint in $N_H(S)$. Notice that $d(u)=k$ for each $u \in X$ and $d(v) \leq k$ for each $v \in Y$. By counting the edges incident to the vertices in $S$, we have that $|E| = k|S|$. On the other hand, by counting the edges incident to the vertices in $N(S)$, we have that $|E| \leq k|N_H(S)|$. It follows that $|S| \leq |N_H(S)|$; so, by Theorem~\ref{thm: Halls}, $H$ has a matching $M$ that saturates $X$. By coloring each vertex $v \in V(G)$ with the color $c \in \mathcal{L}$ to which it is matched in $M$, we obtain a proper $L$-coloring of $G$ with range of size $|V(G)|$.
\end{proof}

The following Lemma from~\cite{KM19} will be used throughout this Subsection. Before we state it, we need a definition. Suppose $G$ is a graph, $L$ is a $k$-assignment for $G$, and $f$ is a proper $L$-coloring of $G$. We say that a color $c \in \mathcal{L}$ is \emph{used excessively} by $f$ if $|f^{-1}(c)| > \lceil \eta(c)/k \rceil$. 

\begin{lem} [\cite{KM19}] \label{lem: excess}
Suppose $G$ is a graph, and suppose that $L$ is a $k$-assignment for $G$ such that $\max_{c \in \mathcal{L}} \eta(c) < 2k.$  If there is a proper L-coloring of $G$ that uses no color $c \in \mathcal{L}$ excessively, then $G$ is proportionally $L$-colorable.
\end{lem}

Lemmas~\ref{lem: leq_k} and~\ref{lem: excess} now allow us to show that $G=K_{n,m}$ is proportionally $L$-colorable whenever $L$ is a $k$-assignment for $G$ with no colors of high multiplicity.

\begin{cor} \label{co: leq_k}
Suppose $n,m \in \N$, $G=K_{n,m}$, and $L$ is a $k$-assignment for $G$ such that for each color $c \in \mathcal{L}$, $\eta(c) \leq k$. Then, there is a proportional $L$-coloring of $G$.
\end{cor}

\begin{proof}
By Lemma~\ref{lem: leq_k}, there is a proper $L$-coloring $f$ of $G$ with range of size $|V(G)|$. This implies that $|f^{-1}(c)| \leq 1$ for each color $c \in \mathcal{L}$. Thus, by Lemma~\ref{lem: excess}, there is a proportional $L$-coloring of $G$.
\end{proof}

We are now ready to take care of list assignments with many colors of high multiplicity.

\begin{lem} \label{lem: upper bound Knm}
Let $m$, $n$, and $d$ be natural numbers such that $m\geq 3d$ and $n\geq 2$. Suppose $G=K_{n,m}$, and $L$ is an $(m+n-d-1)$-assignment for $G$. Let $\alpha$ denote the number of colors of high multiplicity with respect to $L$. If $\alpha \geq (m-d)/2$, then $G$ is proportionally $L$-colorable.
\end{lem}

\begin{proof}
Suppose $G$ has bipartition $A=\{ u_1,\ldots,u_n \}$, $B=\{ v_1,\ldots,v_m \}$. Note that for each color $b \in \mathcal{L}$ exactly one of the following three statements must hold: (1) $0 < \eta(b) < m+n-d-1$ and $b$ must be used either zero times or once in a proportional $L$-coloring of $G$, (2) $\eta(b) = m+n-d-1$ and $b$ must be used exactly once, or (3) $m+n-d-1 < \eta(b) \leq m+n < m + (m - 2d) + n + (n - 2) = 2(m+n-d-1)$ and $b$ must be used either once or twice. Since $m \geq 3d$, $\alpha \geq (m-d)/2 \geq d$. Suppose $c_1,\ldots,c_\alpha$ are $\alpha$ colors of high multiplicity with respect to $L$. Notice that each color of high multiplicity must be in at least $m+n-d$ of the lists associated with $L$. Since $|V(G)|=m+n$, we know that for each $i\in[\alpha]$, there exist at most $d$ vertices $v\in V(G)$ such that $c_i\notin L(v)$. We now prove that $G$ is proportionally $L$-colorable in each of two cases: (1) $m-d$ is even, and (2) $m-d$ is odd.

For the first case, assume without loss of generality that $c_i \in L(v_{2i - 1})$ and $c_i \in L(v_{2i})$ for each $i \in [(m - d) / 2]$. Color the vertices $v_{2i - 1}, v_{2i} \in V(G)$ with $c_i$ for $i \in [(m - d) / 2]$. Now consider the subgraph $G' = G - \{v_1, \ldots, v_{m - d}\}$. Notice that $|V(G')| = m + n - (m - d) = n + d$. Let $L'$ be a list assignment for $G'$ defined by $L'(v) = L(v) - \{c_1, c_2, \ldots, c_{(m - d) / 2}\}$ for each $v \in V(G')$; in particular, notice that $|L'(v)| \geq (m + n - d - 1) - (m - d) / 2 \geq n + d - 1$. Recall that $m + n < 2(m + n - d - 1)$. If $L'$ is not a constant $(n + d - 1)$-assignment for $G'$, then we are able to color each vertex $v \in V(G')$ with a color in $L'(v)$ such that we use $n + d$ distinct colors, which completes a proper $L$-coloring of $G$ that uses no color $c \in \mathcal{L}$ excessively; hence, by Lemma \ref{lem: excess}, we have that $G$ is proportionally $L$-colorable. Thus, we may assume that $L'$ is a constant $(n + d - 1)$-assignment for $G'$, which implies that $L(v) = L(v')$ for each $v, v' \in V(G')$; in particular, $\{c_1, c_2, \ldots, c_{(m - d) / 2}\} \subseteq L(v)$ for each $v \in V(G')$.

We will now consider two subcases: (a) $L$ is a constant $(m + n - d - 1)$-assignment for $G$, and (b) $L$ is not a constant $(m + n - d -1)$-assignment for $G$. In subcase (a), assume without loss of generality that $L(v) = \{1, 2, \ldots, m + n - d - 1\}$ for each $v \in V(G)$. Recall that $n + m < 2(n + m - d - 1)$; as a result, we have that $\lfloor n / 2 \rfloor + \lfloor m / 2 \rfloor \leq (n + m) / 2 < n + m - d - 1$ (i.e., $\lfloor n / 2 \rfloor + \lfloor m / 2 \rfloor + 1 \leq |L(v)|$ for each $v \in V(G)$). Color $u_{2i-1}$ and $u_{2i}$ with $i$ for each $i \in [\lfloor n/2 \rfloor]$, and $v_{2i-1}$ and $v_{2i}$ with $\lfloor n/2 \rfloor + i$ for each $i \in [\lfloor m/2 \rfloor]$. If $n$ and $m$ are both even, then this completes a proper $L$-coloring of $G$ that uses no color $c \in \mathcal{L}$ excessively. If $n$ is odd and $m$ is even, then coloring vertex $u_n$ with $\lfloor n / 2 \rfloor + \lfloor m / 2 \rfloor + 1$ completes a proper $L$-coloring of $G$ that uses no color $c \in \mathcal{L}$ excessively. If $n$ is even and $m$ is odd, then coloring vertex $v_m$ with $\lfloor n / 2 \rfloor + \lfloor m / 2 \rfloor + 1$ completes a proper $L$-coloring of $G$ that uses no color $c \in \mathcal{L}$ excessively. Finally, if $n$ and $m$ are both odd, then we have that for each $v \in V(G)$, $|L(v)| - (\lfloor n / 2 \rfloor + \lfloor m / 2 \rfloor) = m + n - d - 1 - (n - 1) / 2 - (m - 1) / 2 = (n + m) / 2 - d \geq (n+d)/2 > 1$. Thus, by coloring vertex $u_n$ with $\lfloor n / 2 \rfloor + \lfloor m / 2 \rfloor + 1$ and vertex $v_m$ with $\lfloor n / 2 \rfloor + \lfloor m / 2 \rfloor + 2$, we obtain a proper $L$-coloring of $G$ that uses no color $c \in \mathcal{L}$ excessively. Hence, by Lemma \ref{lem: excess}, we conclude that $G$ is proportionally $L$-colorable.

In subcase (b), we have that $L$ is not a constant $(m + n - d - 1)$-assignment for $G$. Then there exists $j \in [(m - d) / 2]$ such that $L(v_{2j - 1}) \neq L(v_m)$ or $L(v_{2j}) \neq L(v_m)$. Recall that $c_i \in L(v_{2i-1})$ and $c_i \in L(v_{2i})$ for $i \in [(m - d) / 2]$. Color $v_{2i - 1}$ and $v_{2i}$ with $c_i$ for $i \in [(m - d) / 2] - \{j\}$, and color $u_1$ and $u_2$ with $c_j$. Consider the subgraph $G'' = G - (\{u_1, u_2, v_1, \ldots, v_{m - d}\} - \{v_{2j - 1}, v_{2j}\})$. Notice that $|V(G'')| = m + n - (m - d) = n + d$.
Let $L''$ be the list assignment for $G''$ defined by $L''(v) = L(v) - \{c_1, c_2, \ldots, c_{(m - d) / 2}\}$ for each $v \in V(G'')$; in particular, notice that $|L''(v)| \geq n + d - 1$ for each $v \in V(G'')$. Recall that $m + n < 2(m + n - d - 1)$. Notice that $L''$ is not a constant $(n + d - 1)$-assignment for $G''$, since that would imply that $L(v_{2j - 1}) = L(v_{2j}) = L(v_m)$. Thus, we are able to color each vertex $v \in V(G'')$ with a color in $L''(v)$ such that we use $n + d$ distinct colors, which completes a proper $L$-coloring of $G$ that uses no color $c \in \mathcal{L}$ excessively. Hence, by Lemma \ref{lem: excess}, we have that $G$ is proportionally $L$-colorable.

Now we consider the case where $m - d$ is odd. Notice that $\alpha \geq (m - d) / 2 + 1 / 2 = (m - d + 1) / 2$. Since $m - d \geq 2d \geq 2$ (and thus $m - d \geq 3$ since $m - d$ is odd), we assume without loss of generality that $c_i \in L(v_{2i - 1})$ and $c_i \in L(v_{2i})$ for $i \in [(m - d + 1) / 2 - 1]$ and $c_{(m - d + 1) / 2} \in L(v_{m - d})$. We now consider two subcases: (a) $c_{(m - d + 1) / 2} \in L(v_{\ell})$ for some $m - d + 1 \leq \ell \leq m$, and (b) $c_{(m - d + 1) / 2} \not\in L(v_{\ell})$ for all $m - d + 1 \leq \ell \leq m$.

In subcase (a), we assume that $c_{(m - d + 1) / 2} \in L(v_{\ell})$ for some $m - d + 1 \leq \ell \leq m$. Without loss of generality, suppose $\ell = m - d + 1$. Color the vertices $v_{2i - 1}, v_{2i} \in V(G)$ with $c_i$ for each $i \in [(m - d + 1) / 2]$. Now consider the subgraph $G''' = G - \{v_1, v_2, \ldots, v_{m - d + 1}\}$. Notice that $|V(G''')| = m + n - (m - d + 1) = n + d - 1$. Let $L'''$ be the list assignment for $G'''$ defined by $L'''(v) = L(v) - \{c_1, \ldots, c_{(m - d + 1) / 2}\}$ for each $v \in V(G''')$; in particular, notice that $|L'''(v)| \geq (m + n - d - 1) - (m - d + 1) / 2 \geq n + d - 3/2$ for each $v \in V(L''')$, which implies that $|L'''(v)| \geq n + d - 1$. Thus, we are able to color each vertex $v \in V(G''')$ with a color in $L'''(v)$ such that we use $n + d - 1$ distinct colors, which completes a proper $L$-coloring of $G$ that uses no color $c \in \mathcal{L}$ excessively. Hence, by Lemma \ref{lem: excess}, we have that $G$ is proportionally $L$-colorable.

In subcase (b), we assume that $c_{(m - d + 1) / 2} \not\in L(v_{\ell})$ for all $m - d + 1 \leq \ell \leq m$. Since $c_{(m - d + 1) / 2}$ is a color of high multiplicity with respect to $L$, this implies that $c_{(m - d + 1) / 2} \in L(u_i)$ for each $i \in [n]$. Color the vertices $v_{2i - 1}$ and $v_{2i}$ with $c_i$ for $i \in [(m - d + 1) / 2 - 1]$, and color $u_1$ and $u_2$ with $c_{(m - d + 1) / 2}$. Consider the subgraph $G^{(4)} = G - \{u_1, u_2, v_1, \ldots, v_{m - d - 1}\}$. Notice that $|V(G^{(4)})| = m + n - (m - d + 1) = n + d - 1$. Let $L^{(4)}$ be a list assignment for $G^{(4)}$ defined by $L^{(4)}(v) = L(v) - \{c_1, \ldots, c_{(m - d + 1) / 2}\}$ for each $v \in V(G^{(4)})$. As in the previous subcase, we have that $|L^{(4)}(v)| \geq n + d - 1$ for each $v \in V(G^{(4)})$. Thus, we are able to color each vertex $v \in V(G^{(4)})$ with a color in $L^{(4)}(v)$ such that we use $n + d - 1$ distinct colors, which completes a proper $L$-coloring of $G$ that uses no color $c \in \mathcal{L}$ excessively. Hence, it follows from Lemma \ref{lem: excess} that $G$ is proportionally $L$-colorable, which concludes our proof.

\end{proof}

Finally, we take care of all remaining list assignments.

\begin{lem} \label{Knm m+n-d-1}
Let $m$, $n$, and $d$ be natural numbers such that $m\geq 3d$ and $n\geq 2$.  Then, $K_{n,m}$ is proportionally $(m+n-d-1)$-choosable. Consequently, $\chi_{pc}(K_{n,m})\leq m+n-d-1$.
\end{lem}

\begin{proof}
Suppose $G=K_{n,m}$, and $G$ has bipartition $A=\{u_1,\ldots,u_n\}$, $B=\{v_1,\ldots,v_m\}$. Note that given any $(m+n-d-1)$-assignment $K$ for $G$ with palette $\mathcal{K}$, for each $b \in \mathcal{K}$ exactly one of the following three statements must hold: (1) $0 < \eta(b) < m+n-d-1$ and $b$ must be used either zero times or once in a proportional $K$-coloring of $G$, (2) $\eta(b) = m+n-d-1$ and $b$ must be used exactly once, or (3) $m+n-d-1 < \eta(b) \leq m+n < m + (m - 2d) + n + (n - 2) = 2(m+n-d-1)$ and $b$ must be used either once or twice. For the sake of contradiction, suppose $G$ is not proportionally $(m+n-d-1)$-choosable. Let $L$ be an $(m+n-d-1)$-assignment for $G$ such that there is no proportional $L$-coloring of $G$ and the number of colors of high multiplicity with respect to $L$, which we call $\alpha$, is as small as possible. By Corollary \ref{co: leq_k}, we know that $\alpha \geq 1$, which implies that there is some color $c\in \mathcal{L}$ such that $\eta(c)\geq m+n-d$. Also, it follows from Lemma \ref{lem: upper bound Knm} that $\alpha \leq (m - d - 1) / 2$.

Since $|V(G)| = m+n$, notice that we have the following bounds: $m+n-d \leq \eta(c) \leq m+n$. Without loss of generality, suppose that there exists $m_c\in \mathbb{N}$ such that $c\in L(v_j)$ if and only if $j \in [m_c]$. Note that $|\{ v_1,\ldots,v_{m_c} \}| \geq m-d$. Suppose $z$ is a color such that $z \notin \mathcal{L}$. Let $L'$ be the $(m+n-d-1)$-assignment for $G$ defined as follows: $L'(v) = (L(v) \cup \{z\})-\{c\}$ for $v \in \{v_1,\ldots,v_{m-d}\}$, and $L'(v) = L(v)$ for $v \in V(G)-\{ v_1,\ldots,v_{m-d} \}$. Suppose that $\mathcal{L}'$ is the palette of $L'$. Since $\eta_{L'}(z) = m-d$ and $\eta_{L'}(c) \leq m+n-(m-d) = n+d < 2d + n \leq m-d+n$, $z$ and $c$ are not colors of high multiplicity with respect to $L'$.  So, there are $\alpha-1$ colors of high multiplicity with respect to $L'$. Thus, we have that $G$ is proportionally $L'$-colorable.

Suppose that $f$ is a proportional $L'$-coloring of $G$. If $|f^{-1}(z)|=0$, then $f$ is a proper $L$-coloring of $G$ that uses no colors excessively, and so Lemma \ref{lem: excess} implies that $G$ is proportionally $L$-colorable which is a contradiction. If $|f^{-1}(c)|=0$ and $|f^{-1}(z)|=1$, then by recoloring the vertex colored with $z$ with the color $c$, we obtain a proportional $L$-coloring of $G$ which is a contradiction.

As such, we may assume $|f^{-1}(c)|=|f^{-1}(z)|=1$, implying that: $f^{-1}(c) \cup f^{-1}(z) \subseteq B$ or $f^{-1}(c) \subseteq A$ and $f^{-1}(z) \subseteq B$. If $f^{-1}(c) \cup f^{-1}(z) \subseteq B$, we can obtain a proportional $L$-coloring of $G$ by recoloring the vertex colored with $z$ with $c$ which is a contradiction. So, we may suppose $f^{-1}(c) \subseteq A$ and $f^{-1}(z) \subseteq B$.  We assume without loss of generality that $f(u_1)=c$ and $f(v_1)=z$. Let $D = \{ a\in \mathcal{L}': |f^{-1}(a)\cap B|=2 \}$, and let $L''(u_1)=L'(u_1)-(f(A) \cup D)$. Notice that $|L'(u_1)| = m+n-d-1$, $|f(A)|\leq n$, and $|D| \leq \alpha - 1 \leq (m - d - 3) / 2$; thus, we have that $|L''(u_1)| \geq (m+n-d-1)-(n+(m-d-3)/2) = (m-d+1)/2 \geq d+1/2$ which implies that $|L''(u_1)| \geq d+1$. Let $S = \{ a\in f(B): a\notin D\cup \{z\} \}$; note that $S$ consists of all the colors other than $z$ that were used by $f$ to color precisely one vertex in $B - \{v_1\}$. As such, we have that $v \in B - \{v_1\}$ if and only if either $f(v) \in D$ or $f(v) \in S$.

Now, we finish the proof by deriving a contradiction in each of the following two cases: (1) $L''(u_1)\not\subseteq S$ and (2) $L''(u_1)\subseteq S$. In case (1), let $L'''(u_1)=L''(u_1)-S$. We know there is a color $w \in L'''(u_1)$. By the definition of $L'''(u_1)$ we know that $f$ did not use $w$. Thus, by recoloring $u_1$ with $w$ and $v_1$ with $c$, we obtain a proportional $L$-coloring of $G$ which is a contradiction.

For case (2), suppose that $L''(u_1) = \{ a_1,\ldots,a_t \}$.  We know that $t \geq d+1$, and since $L''(u_1) \subseteq S$, we further have that $f^{-1}(L''(u_1)) \subseteq B - \{v_1\}$.  Notice that $|f^{-1}(a_i)| = 1$ for each $i \in [t]$, and so $|f^{-1}(L''(u_1))| = t \geq d+1$. Therefore, there must exist a vertex  $v \in f^{-1}(L''(u_1))$ such that $c\in L(v)$ since $c$ is not in at most $d$ of the lists associated with $L$.  As a result, by recoloring $v$ with $c$, $u_1$ with $f(v)$, and $v_1$ with $c$, we obtain a proportional $L$-coloring of $G$ which is a contradiction.
\end{proof}

We can now prove Theorem~\ref{thm: main}.

\begin{proof}
Suppose $G = K_{n,m}$.  We have that $\max\{ n + 1, \lceil n / 2 \rceil + \lceil m / 2 \rceil\} \leq \chi_{pc}(G)$ by Theorem~\ref{thm: old} and Theorem~\ref{pro: tpartitelower}.  The fact that $\chi_{pc}(G) \leq n + m  - 1 - \lfloor m/3 \rfloor $ follows from Theorem~\ref{thm: old} when $m=2$.  When $m \geq 3$, Lemma~\ref{Knm m+n-d-1} implies that $\chi_{pc}(G) \leq n + m - d - 1$ for any $d \in \N$ satisfying $d \leq m / 3$. If we let $d = \lfloor m / 3 \rfloor$, we have that $\chi_{pc}(G) \leq n + m  - 1 - \lfloor m/3 \rfloor$.
\end{proof}

{\bf Acknowledgment.}  The authors would like to thank Hemanshu Kaul for many helpful conversations.

\end{document}